\documentclass[12pt]{amsart}
\usepackage{amssymb}
\usepackage{amsmath}
\usepackage{amsthm}
\usepackage[all]{xy}
\xyoption{2cell}
\usepackage{times}
\usepackage{graphicx} 
\newcommand{\Ol}{{\mathcal O}}
\newcommand{\f}{\varphi}
\newcommand{\E}{{\mathcal E}}
\newcommand{\V}{{\mathcal V}}

\newcommand{\C}{{\mathcal C}}
\newcommand{\cH}{{\mathcal H}}
\newcommand{\El}{{\mathcal L}}
\newcommand{\li}{{\mathcal L}}

\newcommand{\pu}{{\mathbb P^1}}
\newcommand{\pd}{{\mathbb P^2}}
\newcommand{\proj}{\mathbb P}

\DeclareMathOperator{\loc}{\mathrm{Locus}}

\newcommand{\ratcurves}{\textrm{Ratcurves}^n(X)}

\DeclareMathOperator{\codim}{\mathrm{codim}}

\DeclareMathOperator{\cone}{NE}
\DeclareMathOperator{\cycl}{N_1}
\DeclareMathOperator{\pic}{Pic}
\DeclareMathOperator{\Exc}{Exc}

\newcommand{\kd}{-K_X \cdot}

\newcommand{\W}{{\mathcal{W}}}
\newcommand{\conx}[1]{\cone\,(#1,X)}
\newcommand{\cycx}[1]{\cycl(#1,X)}

\newcommand{\rc}[2]{#1 \xymatrix{\ar@{-->}[r] & }{#2}}

\newtheorem{theorem}{Theorem}[section]
\newtheorem{lemma}[theorem]{Lemma}

\newtheorem{proposition}[theorem]{Proposition}

\newtheorem{example}[theorem]{Example}
\theoremstyle{definition}
\newtheorem{definition}[theorem]{Definition}

\theoremstyle{remark}
\newtheorem{remark}[theorem]{Remark}

\begin{document}
\title{Rationally cubic connected manifolds II}

\renewcommand{\theequation}{{\arabic{section}.\arabic{theorem}.\arabic{equation}}
}
\author{Gianluca Occhetta}
\author{Valentina Paterno}

\address{Dipartimento di Matematica, via Sommarive 14, I-38123 Povo (TN)}

\email{gianluca.occhetta@unitn.it}
\email{paterno@science.unitn.it}

\thanks{Partially supported by  MIUR (PRIN project:
Propriet\`a
geometriche delle variet\`a reali e complesse)}


\maketitle

\begin{abstract} We study smooth complex projective polarized varieties $(X,H)$ of dimension $ n \ge 2$ which admit a dominating family $V$ of rational curves of $H$-degree $3$, such that two general points of $X$ may be joined by a curve parametrized by $V$ and which do not admit a covering family of lines (i.e. rational curves of $H$-degree one).
We prove that such manifolds are obtained from RCC manifolds of Picard number one by blow-ups along smooth centers.\\ If we further assume that $X$ is a Fano manifold, we obtain a stronger result, classifying all Fano RCC manifolds of Picard number $\rho_X \ge 3$.\end{abstract}

\section{Introduction}

In our recent paper \cite{OP}, inspired by the classification of conic-connected manifolds given in \cite{IR}, we started the study of  \emph{rationally cubic connected} manifolds (RCC-manifolds, for short), i.e. smooth complex projective polarized varieties $(X,H)$ of dimension $ n \ge 2$ which are rationally connected by rational curves of degree $3$ with respect to a fixed ample line bundle $H$.\\
In \cite{OP} we considered manifolds covered by lines (i.e. rational curves of degree one with respect to $H$), proving that the Picard number of such manifolds is at most three and that if equality holds the manifold has an adjunction scroll structure over a smooth variety.\par
\medskip
In the present paper we will complete our task by dealing with RCC-manifolds not covered by lines. Our main result shows that such manifolds are obtained by RCC-manifolds of Picard number one by blow-ups along smooth centers. More precisely we have the following:

\begin{theorem}\label{main} Let $(X,H)$ be a RCC-manifold with respect to $V$, not covered by lines. Then there exists a polarized manifold $(X',H')$  of Picard number one not covered by lines and a contraction $\f_\Sigma:X \to X'$ expressing $X$ as a blow-up of $X'$ along disjoint centers $T_i$ with exceptional divisors $E_i$ such that  $H \simeq \f_\Sigma^*H' - \sum E_i$. The pair $(X',H')$ is RCC with respect to $V'$, the family of deformations of the image of a general curve parametrized by $V$.\\
Moreover either $(X',H')\! \simeq\!  (\proj^n, \Ol_{\proj}(3)) $ or $\pic(X')=\langle H' \rangle $ and $-K_{X'}\!= \frac{n+1}{3}H'$.\linebreak  In the latter case we have $\dim T_i < (n+1)/3$ for every $i$.
\end{theorem}

First of all we deal with the case of RCC-surfaces, giving a complete classification of them in Section 3, then we move to the general case.\\
The main idea is the following: if the Picard number of $X$ is greater than one, then for every point $x \in X$ the cubics parametrized by $V$ passing through $x$ must degenerate into reducible cycles whose components are not numerically proportional to $V$. Clearly these degenerations can be into cycles consisting either of three lines or of a line and a conic. Since we are assuming that $X$ is not covered by lines, the latter happens for a general point, and moreover the irreducible component through the general point is the conic, so, for each possible degeneration we get a dominating family of conics and a family of lines.\\
It turns out that the loci of the families of lines arising in this way are divisors, which we will call {\em divisors of $V$-lines};  the main point of the proof is to show that these divisors are (disjoint) exceptional divisors of smooth blow-ups and that they do not meet a general cubic, so that, after blowing them down, we still have a RCC-manifold.\\
The hardest case is that of manifolds of Picard number two, which is treated in Section \ref{secpic2}, while the general one is settled in Section \ref{secgen}.\par
\smallskip
If we further assume that the manifold is Fano, the results are much stronger and for $\rho_X \ge 3$ we have a complete classification:

\begin{theorem}\label{th:15} Let $(X,H)$ be a polarized Fano manifold of dimension $n >2$ and $\rho_X \ge 3$. Suppose that $X$ is RCC with respect to a family $V$ and doesn't admit a covering family of lines. Then  $\rho_X=3$, $X$ has a contraction $\pi:X \to \proj^n$ which is the blow-up of $\proj^n$ along two disjoint centers $T_1$ and $T_2$ which can be: 
\begin{enumerate}
\item [(1)] two linear spaces $\Lambda_1,\Lambda_2$ such that
$$\begin{array}{ll}
\Lambda_1\cap\Lambda_2=\emptyset,&
\dim \Lambda_1+\dim \Lambda_2=n-2,
\end{array}$$	
	\item [(2)]  a linear space $\Lambda_1$ and  a smooth quadric $Q_1\subset\Lambda_2\simeq\mathbb{P}^{\dim Q_1+1}$ such that
$$\begin{array}{lll}
\Lambda_1\cap\Lambda_2=\emptyset,&
 \dim Q_1\geq\frac{n}{2}-1,&\dim \Lambda_1+\dim Q_1=n-2,
\end{array}$$	
	\item [(3)] two smooth quadrics  $Q_1\subset\Lambda_1\simeq\mathbb{P}^{\frac{n}{2}}$ and $Q_2\subset\Lambda_2\simeq\mathbb{P}^{\frac{n}{2}}$ such that 
$$\begin{array}{lll}
Q_1 \cap Q_2 = \emptyset,&
\dim \Lambda_1\cap\Lambda_2=0,&
\dim Q_1=\dim Q_2=\frac{n}{2}-1,
\end{array}$$		
\end{enumerate}
and, denoted by $E_1$ and $E_2$ the exceptional divisors,  $H \simeq \pi^*\mathcal{O}_{\mathbb{P}^n}(3)-E_1-E_2$.
\end{theorem}

Section \ref{secfano} is devoted to the proof of this theorem, and to some other considerations on RCC Fano manifolds of Picard number two.\\
In the last section we examine in detail the manifolds appearing in Theorem (\ref{th:15}), describing their cone of curves, their contractions and their families of rational curves.


\section{Background material}


We gather in this section the basic definitions and results regarding families of rational curves
and Mori theory that we are going to use.

\subsection{Families of rational curves and of rational $1$-cycles}\par
\medskip

\begin{definition} \label{Rf}
A {\em family of rational curves} $V$ on $X$ is an irreducible component
of the scheme $\ratcurves$ (see \cite[Definition II.2.11]{Kob}).\\
Given a rational curve we will call a {\em family of
deformations} of that curve any irreducible component of  $\ratcurves$
containing the point parametrizing that curve.\\
We define $\loc(V)$ to be the set of points of $X$ through which there is a curve among those
parametrized by $V$; we say that $V$ is a {\em covering family} if ${\loc(V)}=X$ and that $V$ is a
{\em dominating family} if $\overline{\loc(V)}=X$.\\
By abuse of notation, given a line bundle $H \in \pic(X)$, we will denote by $H \cdot V$
the intersection number $H \cdot B$, with $B$ any curve among those
parametrized by $V$.
We will say that $V$ is {\em unsplit} if it is proper; clearly, an unsplit dominating family is
covering.\\
We denote by $V_x$ the subscheme of $V$ parametrizing rational curves
passing through a point $x$ and by $\loc(V_x)$ the set of points of $X$
through which there is a curve among those parametrized by $V_x$. If, for a general point $x \in \loc(V)$,
$V_x$ is proper, then we will say that the family is {\em locally unsplit}.
Moreover, we say that $V$ is {\em generically unsplit} if the fiber of the double-evaluation map
$$\begin{array}{rllc}
\Pi:&V&\rightarrow& X\times X\\
&[f]&\mapsto&(f(q),f(p))
\end{array}$$
over the general point of its image has dimension  $0$.\end{definition}

\begin{definition}\label{CF}
We define a {\em Chow family of rational 1-cycles} $\W$ to be an irreducible
component of  $\textrm{Chow}(X)$ parametrizing rational and connected 1-cycles.\\
If $V$ is a family of rational curves, the closure of the image of
$V$ in $\textrm{Chow}(X)$, denoted by $\V$, is called the {\em Chow family associated to} $V$.
If $V$ is proper, { i.e.} if the family is unsplit, then $V$ corresponds to the normalization
of the associated Chow family $\V$.
\end{definition}

\begin{definition}
Let $V$ be a family of rational curves and let $\V$ be the associated Chow family. We say that
$V$ (and also $\V$) is {\em quasi-unsplit} if every component of any reducible cycle parametrized by 
$\V$ has numerical class proportional to the numerical class of a curve parametrized by $V$.
\end{definition}

\begin{definition} We say that $k$ quasi-unsplit families $V^1, \dots, V^k$ are numerically independent
if in $\cycl(X)$ we have $\dim \langle [V^1], \dots, [V^k]\rangle =k$.
\end{definition}

For special families of rational curves  we have useful dimensional estimates. The basic one is the following:

\begin{proposition} (\cite[Corollary IV.2.6]{Kob}\label{iowifam})
Let $V$ be a family of rational curves on $X$ and $x \in \loc(V)$ a point such that every component of $V_x$ is proper.
Then
  \begin{itemize}
       \item[(a)] $\dim \loc(V)+\dim \loc(V_x) \ge \dim X  -K_X \cdot V -1$;
       \item[(b)] every irreducible component of $\loc(V_x)$ has dimension $\ge -K_X \cdot V -1$.
    \end{itemize}
\end{proposition}

\begin{definition}
Let $V^1, \dots, V^k$ be families of rational curves on $X$ and $Z \subset X$.
We define $\loc(V^1)_Z$ to be the set of points $x \in X$ such that there exists
a curve $C$ among those parametrized by $V^1$ with
$C \cap Z \not = \emptyset$ and $x \in C$. We inductively define
$\loc(V^1, \dots, V^k)_Z := \loc(V^k)_{\loc(V^1, \dots,V^{k-1})_Z}$.
\end{definition}

{\bf Notation}: 
If $\Gamma$ is a $1$-cycle, then we will denote by $[\Gamma]$ its numerical equivalence class
in $\cycl(X)$; if $V$ is a family of rational curves, we will denote by $[V]$ the numerical equivalence class
of any curve among those parametrized by $V$.
A proper family will always be denoted by a calligraphic letter.\\
If $Z \subset X$, we will denote by $\cycx{Z} \subseteq \cycl(X)$ the vector subspace
generated by numerical classes of curves of $X$ contained in $Z$; moreover, we will denote 
by  $\conx{Z} \subseteq \cone(X)$
the subcone generated by numerical classes of curves of $X$ contained in $Z$. We will
denote by $\langle \dots \rangle$ the linear span.\par
\medskip
We will use some properties of $\loc(V)_Z$, summarized in the following

\begin{lemma}\label{locvf}\cite[Section 5]{ACO}, \cite[Proof of Lemma 1.4.5]{BSW}
Let $Z \subset X$ be an irreducible closed subset and $\V$ an unsplit family.
Then every curve contained in $\loc(\V)_Z$ is numerically equivalent to
a linear combination with rational coefficients
   $$\lambda C_Z + \mu C_{\V},$$
where $C_Z$ is a curve in $Z$, $C_{\V}$ is a curve among those parametrized by $\V$ and $\lambda \ge 0$.
If moreover curves contained in $Z$ are numerically independent from curves in $\V$ and 
$Z \cap \loc(\V) \not= \emptyset$ then
$$\dim \loc(\V)_Z \ge \dim Z -K_X \cdot \V - 1.$$
\end{lemma}

We will also need the following lemma, which is based on \cite[Proposition II.4.19]{Kob}:

\begin{lemma}\label{numeq}
Let $Y \subset X$ be a closed subset, $\V$ a Chow family of rational $1$-cycles. Then every curve
contained in $\loc(\V)_Y$ is numerically equivalent to a linear combination with rational
coefficients of a curve contained in $Y$ and of irreducible components of cycles parametrized
by $\V$ which meet $Y$.
\end{lemma}

\subsection{Contractions and fibrations}\par
\medskip

\begin{definition}
Let $X$ be a manifold such that $K_X$ is not nef.
By the  Cone Theorem
 the closure of the cone of effective 1-cycles into
the $\mathbb R$-vector space of 1-cycles modulo numerical equivalence,
$\overline {\cone}(X) \subset\cycl(X)$, is polyhedral in the part contained in
the set $\{z\in \cycl(X) :K_X \cdot z<0\}$.
An {\em extremal face}  is a face of this polyhedral part, and an extremal face of dimension one
is called an {\em extremal ray}.\\
To an extremal face $\sigma \subset \cone(X)$ is associated a morphism with connected fibers 
$\f_\sigma:X \to Z$ onto a normal variety, morphism which contracts the curves 
whose numerical class is in $\sigma$; $\f_\sigma$ is called an {\em extremal contraction}
or a {\em Fano-Mori contraction}, while a Cartier divisor $H$ such that 
$H = \f_\sigma^*A$ for an ample divisor $A$ on $Z$
is called a {\em supporting divisor} of the map $\f_\sigma$ (or of the face $\sigma$).
We denote with $\Exc(\f_{\sigma}):=\{x\in X|\dim \f_\sigma ^{-1}(\f_{\sigma}(x))>0\}$ the {\em exceptional locus} of $\f_\sigma$.\\
An extremal contraction associated to an extremal ray is called an 
{\em elementary contraction};
an elementary contraction is said to be of {\em fiber type} if $\dim X>\dim Z$, otherwise the contraction is {\em birational}. Moreover, if the codimension of the exceptional locus of an elementary birational contraction is equal to one, then the contraction is called {\em divisorial}; otherwise it is called {\em small}.
\end{definition}

Proposition (\ref{iowifam}), in case $\mathcal{V}$ is the unsplit family of deformations of a minimal extremal
rational curve, gives the {\em fiber locus inequality}:

\begin{proposition} \cite{Io,Wicon}\label{fiberlocus} Let $\f$ be a Fano-Mori contraction
of $X$ and let $E = \Exc(\f)$ be its exceptional locus;
let $S$ be an irreducible component of a (non trivial) fiber of $\f$. Then
$$\dim E + \dim S \geq \dim X + l -1,$$
where $l =  \min \{ -K_X \cdot C\ |\  C \textrm{~is a rational curve in~} S\}.$
If $\f$ is the contraction of a ray $R$, then $l(R):=l$ is called the {\em length of the ray}.
\end{proposition}

The next theorem, which will be frequently used, gives us conditions to ensure that a birational contraction is a smooth blow-up: 

\begin{theorem}(Cf. \cite[Theorem 4.1 (iii)]{AW1})\label{blowup} Let $\f:X\rightarrow Z$ be an extremal contraction of a smooth variety $X$. Assume that $\f$ is birational and supported by $K_X+rH$, with $H$ a $\f$-ample line bundle on $X$, and that, for each non trivial fiber $F$ of $\f$ we have $\dim F=r$.
Then $Z$ is smooth and $\f$ is a blow down of a smooth divisor $E\subset X$ to a smooth subvariety of $Z$.
\end{theorem}


\section{Preliminaries}


\begin{definition} Let $(X,H)$ be a polarized manifold of dimension $n$; if there exists a dominating family $V$ of rational curves such that $H \cdot V=3$ and through two general points of $X$ there is a curve
parametrized by $V$  we will say that $X$ is {\em Rationally Cubic Connected} - RCC for short - with respect to $V$.
\end{definition}

In \cite[Proposition 4.3]{OP} we proved the following result concerning RCC manifolds of Picard number one:

\begin{proposition}\label{pic1}
Let $(X,H)$ be a RCC-manifold with respect to a family $V$; then 
\begin{enumerate}
\item there exists $x \in X$ such that $V_x$ is proper if and only if $(X,H) \simeq (\proj^n,\Ol(3))$;
\item there exists $x \in X$ such that $V_x$ is quasi-unsplit if and only if $X$ is a Fano manifold of Picard number one and index $r(X) \ge \frac{n+1}{3}$ with fundamental divisor $H$.
\end{enumerate}
\end{proposition}

As a consequence, if the Picard number of $X$ is greater than one, through
a general point there exists  at least one reducible cycle parametrized by the Chow family $\V$ whose components are not all numerically proportional to $V$.\\
Since $H \cdot V=3$, such a cycle in $\V$ can have two or three irreducible rational components;
we will call a component of $H$-degree one a \textit{line} and a component of $H$-degree two a \textit{conic}.\\
Throughout the present paper, after dealing with the case of surfaces in the next subsection, we will assume that $X$ does not admit a covering family of lines; this assumption yields that every dominating family of conics is locally unsplit.\\
Moreover, from \cite[Formula (1), Corollary 5.6 and Formula (4)]{OP} we have that $V$ is generically unsplit and so
\begin{equation}\label{n+1}
-K_X \cdot V= n+1.
\end{equation}

\subsection{RCC surfaces}

We will now give a complete classification of RCC-surfaces; as a consequence we will
see that Theorem (\ref{main}) holds for $n=2$.

\begin{proposition}\label{surfaces} Let $(S,H)$ be a polarized surface, which is RCC with respect to a family 
of rational curves $V$. Then $(S,H)$ and $V$ are one of the following:
\begin{enumerate}
\item $(\proj^2, \Ol_{\proj^2}(3))$, the family of lines in $\pd$;
\item $(\proj^2, \Ol_{\proj^2}(1))$,  the family of rational plane cubics;
\item $(\mathbb Q^2, \Ol_{\mathbb Q^2}(1,2)$ (or $\Ol_{\mathbb Q^2}(2,1)))$, the family of curves of type $(1,1)$;
\item $(\mathbb Q^2, \Ol_{\mathbb Q^2}(1,1))$, the family of curves of type $(2,1)$ (or $(1,2)$);
\item $(\mathbb F_1,C_0 +3f)$, the family of curves of type $C_0 +f$;
\item $(S_k,-K_{S_k})$ with $S_k$ a blow-up of $\pd$ in $k$ general points, with $k=1, \dots,8$, the family of strict transforms of lines in $\pd$.
\end{enumerate}

\end{proposition}

\begin{proof}
By the first part of Proposition (\ref{pic1}) if $V_x$ is proper for some $x \in S$ then $(S,H) \simeq (\pd,\Ol_\pd(3))$, while, by the second part, if $V_x$ is quasi-unsplit for some $x \in S$
then $(S,H) \simeq (\pd,\Ol_\pd(1))$.\par
\smallskip
We can assume from now on that the Picard number of $S$ is at least two.\\
If $V$ is not generically unsplit then either 
$(S,H)=(\mathbb Q^2, \Ol_{\mathbb Q^2}(1,2))$ (and curves of $V$ are curves of type $(1,1)$) or through every pair of points of $S$ there is a reducible cycle in $V$ consisting of three lines by \cite[Proposition 5.5]{OP}; in this last case $S$ admits two dominating families of lines, hence $(S,H) \simeq (\mathbb Q^2, \Ol(1,1))$ and $V$ is, up to exchange the rulings,  the family of curves of type $(2,1)$.\par
\smallskip
We are thus left with the case of $V$ generically unsplit; by formula (\ref{n+1}\!\!) we then
have $-K_S \cdot V=3$.\par
\smallskip
We consider first the case in which $S$ admits a covering family of lines $\li$; recalling that $\rho_S \ge 2$ we have that $S$ is a ruled surface $\mathbb F_e=\proj_\pu(\Ol(-e) \oplus \Ol)$, and the lines are the fibers of the projection to $\pu$.\\
 Denote by $C_0$ a minimal section and by $f$ a fiber;
since the fibers are lines with respect to $H$ we can write $H \equiv C_0 +hf$. Let $B \equiv aC_0+bf$ be a curve parametrized by $V$; by the genus formula we get 
$$1= B^2= a(2b -ae),$$ hence $a=1$ and $e=2b-1$. Since $B$ is an effective curve, this is possible just for $b=e=1$.
Now, from $H \cdot B=3$, we get $h=3$.\par
\smallskip
Finally we treat the case of a surface not covered by lines.\\
Consider the set $\mathcal B'=\{(\El^i, C^i)\}$ of pairs of families $(\El^i, C^i)$ such that  through a general point $x \in S$ there  is a reducible cycle $\ell +  \gamma$ belonging to $\V$, with $\ell$ and $\gamma$ parametrized respectively by $\El^i$ and $C^i$. \\
The families of conics are locally unsplit and dominating, hence $-K_S \cdot C^i=2$ for every $i$; this implies that $-K_S \cdot \li^i=1$. The numerical classes of $\li^i$ and $C^i$ are a system of generators for $N_1(S)$ (by Lemma (\ref{numeq})), and $K_S +H$ is trivial on each of them.\\  
It follows that  $H$ and $-K_S$ are numerically equivalent;  being $S$ rational they are also linearly equivalent. In particular $-K_S$ is ample, and $S$ is a del Pezzo surface. We are now assuming that the Picard number of $S$ is greater than one, and that  $-K_S \cdot V=3$, so $S$ is not a projective space or a quadric.
\end{proof}


\section{Divisors of $V$-lines}


Having settled the case of surfaces in Proposition (\ref{surfaces}), we will assume from now on that $n=\dim X \ge 3$.\par
\smallskip
Consider the set $\mathcal B'=\{(\El^i, C^i)\}$ of pairs of families $(\El^i, C^i)$ such that  through a general point $x \in X$ there  is a reducible cycle $\ell +  \gamma$, parametrized by $\V$, with $\ell$ and $\gamma$ parametrized respectively by $\El^i$ and $C^i$. \par
\medskip
Let us consider two pairs $(\li^i,C^i)$ and $(\li^j,C^j)$ such that $[\li^i] \not = [\li^j]$.\\
Since no family of lines is covering, by the generality of $x$ all the families of conics are dominating and locally unsplit;
therefore $\dim \loc(C^i)_x \cap \loc(C^j)_x = 0$ for every $i \not = j$;
it follows that 
\begin{equation}\label{limc}
 -K_X \cdot (C^i+C^j) \le \dim \loc(C^i)_x + \loc(C^j)_x +2 \le n +2, 
\end{equation}
so, recalling that $-K_X \cdot (C^i+\li^i) = \kd V = n+1$, we also have
\begin{equation}\label{liml}
 -K_X \cdot (\li^i+\li^j) \ge n.
\end{equation}

Let now $\mathcal B=\{(\El^i, C^i)\}_{i=1}^k$ be a maximal set of pairs in $\mathcal B'$ 
with the property that $[V],  [\El^1], \dots,  [\El^k]$ are numerically independent. Denote by $\Pi_i$ the two-dimensional vector subspace of
$\cycl(X)$ spanned by $[V]$ and $[\li^i]$.
By Lemma (\ref{numeq}) we have 
$$\cycl(X) = \langle [V], [\El^1],[C^1], \dots, [\El^k],[C^k] \rangle= \langle [V], [\El^1], [\El^2], \dots,[\El^k] \rangle,$$
 hence the Picard number of $X$ is $k+1$.\par
 \medskip
For every $i= 1, \dots, k$ denote by $E_i$ the set $\loc(C^i, \El^i)_x$;
by Lemma (\ref{locvf}) and by Proposition (\ref{iowifam}) it has dimension $\dim E_i \ge n-1$; since 
 $E_i \subset \loc(\El^i)$, the inclusion is an equality and $E_i$ is an irreducible divisor.\\
We will call the divisor $E_i$ the {\em divisor of $V$-lines} associated to the pair $(\li^i,C^i)$.
We will distinguish the divisors of $V$-lines in the following way:
\begin{enumerate}
\item $E_i$ is of the first kind if $-K_X \cdot \El^i = n-1$;
\item $E_i$ is of the second kind if $-K_X \cdot \El^i = 1$;
\item $E_i$ is of the third kind in all the other cases.
\end{enumerate}

The next lemma gives the description of the relative space of cycles of the divisors of $V$-lines.
 
\begin{lemma}\label{vdiv} 
Let $E_i$ be a divisor of $V$-lines, associated to a pair $(\li^i,C^i)$. If
$E_i$ is of the first kind then $\cycx{E_i}=\langle [\li^i] \rangle$. In the other cases
$\cycx{E_i}=\langle [C^i], [\El^i] \rangle$ and  $[\El^i]$ is extremal in $\conx{E_i}$.
\end{lemma}

\begin{proof}
If $E_i$ is of the first kind, then $E_i= \loc(\El^i)_x$ for any $x \in \loc(\El^i)$, while if  $E_i$ is either of the second or of the third
kind then $E_i=\loc(C^i, \El^i)_x$ for a general $x \in X$. The statement now follows from Lemma (\ref{locvf}).
\end{proof}

As a consequence, we can prove that these divisors are disjoint:

\begin{lemma}\label{disj}
Let $E_i$ for $i = 1 , \dots, k$ be divisors of $V$-lines associated to pairs in $\mathcal B$.
Then the $E_i$'s are pairwise disjoint. Moreover, if $E_k$ is of the third kind then
$E_i \cdot \li^k = E_i \cdot C^k = 0$ for every $i \not = k$.
\end{lemma}

\begin{proof}
Since we are assuming that $n \ge 3$, if two divisors met, their intersection should be positive
dimensional. Therefore, by the description of the relative space of cycles $\cycx{E_i}$, it is clear that the divisor of the first kind are disjoint from any other divisor. Moreover, if a divisor of the second kind exists, then, by  equation (\ref{limc}\!), all the other divisors are of the first kind.\\
We will now show that, if $E_k$ is of the third kind then $E_i \cdot \li^k = E_i \cdot C^k = 0$ for every $i \not = k$. This implies also that two divisors of the third kind are disjoint.\\
Both $\dim \loc(C^k)_x$ and $\dim \loc(\li^k)_x$ are greater than one, so if $E_i \cdot C^k >0$
(respectively $E_i \cdot \li^k >0$) then $E_i$ would contain a curve whose numerical class
is proportional to $[C^k]$ (resp. $[\li^k]$), a contradiction, since neither $[C^k]$ nor $[\li^k]$ is contained in $\cycx{E_i}$. 
\end{proof}

Theorem (\ref{main}) will follow if we prove that all the divisors of $V$-lines  have intersection number zero with $V$. In fact we have the following:

\begin{proposition}\label{key} Let $\mathcal F=\{E_1, \dots, E_k\}$ be a collection of pairwise disjoint  divisors of $V$-lines such that $E_i \cdot V=0$ for every $i=1, \dots, k$.\\
Then there exist a polarized manifold $(X',H')$  not covered by lines and a contraction $\f_\sigma:X \to X'$ expressing $X$ as a blow-up of $X'$ along $k$ disjoint centers $T_i$, with exceptional divisors $E_1, \dots, E_k$ and such that $H = \f_\sigma^*H' - \sum E_i$.\\ Moreover $(X',H')$ is RCC with respect to $V'$, the family of deformations of the image of a general curve parametrized by $V$.
\end{proposition}

\begin{proof} The effective divisor $E_i$ cannot be trivial on the whole $\cone(X)$; since it vanishes on $[V]$, which lies in the interior of $\cone (X)$ it must be negative on some effective curve $B$. This curve is therefore contained in $E_i$.\\
By Lemma (\ref{vdiv}) the numerical class of $B$ is contained in the two-dimensional vector subspace of
$\cycl(X)$ spanned by $[C^i]$ and $[\li^i]$; since $E_i \cdot C^i \ge 0$, being $C^i$ a dominating family, and $E_i \cdot V=0$ we have $E_i \cdot \li^i <0$.\\
Consider the divisor $H_i=  -(E_i \cdot \li^i)H + E_i$; we will show that this divisor is nef 
and trivial only on $R^i= \mathbb R_+[\li^i]$. Assume that, for some curve $B$ we have $H_i \cdot B \le 0$; this implies $E_i \cdot B <0$, so $B \subset E_i$,  hence $[B] \subset \conx{E_i}$.\\
Recalling that $[\li^i]$ is extremal in $\conx{E_i}$, it is clear that for every curve whose numerical class is in  $\conx{E_i}\subset \Pi_i$ the intersection number with $H_i$ is nonnegative, and it is zero if and only if $[B] \in  R^i$; hence $H_i$ is nef and $R^i$ is an extremal ray of $\cone(X)$.\\
Denote by $\f_i$ the contraction associated to $R^i$. Since $\loc(\li^i) = E_i$ and $E_i \cdot R^i <0$ then $\Exc(\f_i)= E_i$; moreover, being $E_i = \loc(\li^i)_{\loc(C^i)_x}$ for a general $x \in X$ any fiber $F_i$  of  $\f_i$ meets $\loc(C^i)_x$, hence
$$n \ge \dim F_i + \dim \loc(C^i)_x \ge \kd \li^i \kd C^i -1 = n.$$
Equality must then hold. In particular for any fiber of $\f_i$ we have $\dim F_i = \kd \li^i$, thus $\f_i$ is a smooth blow-up by Theorem (\ref{blowup}). Notice that from this it follows that $E_i \cdot \li^i =-1$.\par
\smallskip
Consider now the divisor $ H + \sum E_i$; arguing as we did for $H_i$ we prove that it is nef and it vanishes only on curves
whose numerical class belong to one of the $R^i$'s. Therefore there is a $k$-dimensional
face $\sigma$ of $\cone(X)$  generated by the $R^i$'s and the associated contraction $\f_\sigma:X \to X'$
contracts exactly the curves whose numerical class belongs to $R^i$ for some $i$. Since the $E_i$'s are disjoint $\f_\sigma$ is the blow-up of $X'$ along smooth disjoint centers $T_i$'s.\par
\smallskip
Let $V'$ be a family of deformations of the image of a general curve
parametrized by $V$; clearly through two general points of $X'$ there is a curve parametrized by $V'$.  
The divisor  $H + \sum E_i$ is nef and  supports the face contracted by $\f_\sigma$, hence there exists an ample divisor $H'$ on $X'$ such that $\f_\sigma^*H'=H + \sum E_i$. From the projection formula we get $H' \cdot V'=3$.\\
Assume by contradiction that $X'$ is covered by lines, i.e. there exists a dominating family of rational curves $\li'$ of degree one with respect to $H'$. The family $\li$ of deformations of the strict transform of a general line $\ell'$ parametrized by $\li'$ will be a covering family of lines for $X$; in fact, being general, $\ell'$ is disjoint from $T_i$ for every $i$, hence $H \cdot \li=1$.
\end{proof}


\section{Manifolds of Picard number two}\label{secpic2}


In this section we are going to prove the first part of Theorem (\ref{main}) under the assumption that the Picard number of $X$ is two. This is the hardest case and represents a crucial step in the proof.\\
Let $\mathcal B'=\{(\El^i, C^i)\}$ be  as in the previous section; as we saw,
to each pair in $\mathcal B'$ is associated a divisor $E_i = \loc(\li^i)$; we need to show that one of them is the exceptional divisor of a smooth blow-up and does not meet a general curve of $V$.\par
\smallskip
We deal first with a particular case, namely the case in which a second kind divisor of $V$-lines exists.

\begin{proposition}\label{second}
Let $(X,H)$ be a RCC-manifold with respect to $V$, not covered by lines and of Picard number two. Assume that there exists a divisor of $V$-lines $E$ which is of the second kind. Then there exists a contraction $\f:X \to \proj^n$ expressing $X$ as a blow-up of $\proj^n$ along a codimension two linear subspace or along a codimension two smooth quadric. Moreover $H = \f^*\Ol_{\proj^n}(3)-E$ and $V$ is the family of deformations of the strict transform of a general line in $\proj^n$.
\end{proposition}

\begin{proof} Let $(\li,C)$ be the pair in $\mathcal B'$ whose associated divisor of $V$-lines is $E$.\\ 
The two cases appearing in the statement differ by the position in $\cone(X)$ of the numerical class $[C]$.
Assume first  that  $[C]$ spans an extremal ray of $\cone(X)$. \\
The associated contraction $\psi:X \to B$ is then of fiber type with fibers of dimension $n-1$. In fact a fiber $F$ contains $\loc(C_x)$ for every $x \in F$ and, for a general $x \in X$ we have $\dim \loc(C_x) \ge n-1$ since $E$ is a second kind divisor; a general fiber of $\psi$ is a projective space by \cite[Theorem 3.6]{Kesing}; since the contraction is elementary, by standard arguments we get that  $X$ is a projective bundle over $B$; since $X$ is rationally connected we have that $B$ is rational, and thus $X=\proj_{\pu}(\E)$ with $\E = \oplus \Ol(a_i)$ and $0=a_0 \le a_1 \le \dots \le a_n$.\\
The family $C$ is the family of lines in the fibers of $\psi$; recalling that $H \cdot C=2$ and denoted by $\xi_\E$ the tautological line bundle of $\E$ we can write $H = 2 \xi_\E +\psi^*\Ol(b)$ for some $b$.
The ampleness of  $H$ yields $b\ge 1$; equality holds, since $H \cdot \li=1$.
Moreover from the last formula we get that a curve of $\li$ is a section corresponding to a surjection $\E \to \Ol$. The locus of curves in $\li$ is a divisor, hence we have $a_0=a_1= \dots =a_{n-1} =0$. Finally, from $-K_X \cdot \li=1$ we get $a_n=1$.
\par
\smallskip
Assume now that  $[C]$ is not extremal in $\cone(X)$.  Since for a general $x \in X$ we have $\dim \loc(C_x) =n-1$ by Proposition (\ref{iowifam}), in view of \cite[Theorem 2]{BCD} this implies that $C$ is not a quasi-unsplit family.\\
Let $\ell^1 + \ell^2$ be a reducible cycle in $\C$ whose components are not numerically proportional to $C$. For a general $x \in X$ we have seen that $\loc(C_x)$ is a divisor $D_x$; $D_x$ cannot contain curves numerically proportional to $\ell^i$, hence, if $D_x \cdot \ell^i \not = 0$ then, for every point $y$ in the locus of the corresponding family $\li^i$, we have $\dim \loc(\li^i_y) =1$.\\
Assume, up to exchange indexes, that $D_x \cdot \ell^1 \not = 0$; then, since $\li^1$ is not a covering family, by Proposition (\ref{iowifam}) we have $-K_X \cdot \li^1=1$, and thus $-K_X \cdot \li^2=n-1$.\\
By Lemma (\ref{locvf}) $\cycx{D_x}=\langle [C] \rangle$; in particular, being $C$ a dominating family ${D_x}_{|D_x}$ is nef. Since $D_x$ is effective, it follows that it is nef.\\
The nef divisor $D_x$ is trivial on $E_2=\loc(\li^2_x)$, hence $[\li^2]$ generates an extremal ray, which is birational and of length $n-1$, so it corresponds to the blow-up of a smooth point in a smooth $X'$.\\
Let $W$ be a minimal dominating family of rational curves for $X'$, and let $W^*$ be the family of deformations of the strict transform of a curve in $W$; we have that $[W^*]=[C]$, since $E_2$ is trivial on both and $X$ does not carry a covering family of lines. Therefore
\begin{equation}\label{km+1}
-K_{X'} \cdot W =-K_X \cdot W^* =-K_X \cdot C= n
\end{equation}
and $X'$ is a smooth quadric by \cite[Theorem 0.1, (3)]{Miqu}; in particular $X$ has another contraction whose exceptional locus is $E$, which is the blow-up of $\proj^n$ along a smooth quadric of codimension $2$.
\end{proof}

Now we will show that,  up to numerical equivalence, $\mathcal B'$ contains only one pair.

\begin{proposition}\label{proprho2} Let $(X,H)$ be a RCC-manifold with respect to $V$, not covered by lines, and of Picard number two. Then, up to numerical equivalence, $\mathcal B'$ contains only one pair $(\li,C)$.\end{proposition}

\begin{proof}
We will prove the proposition by contradiction.\\
By Proposition (\ref{second}) we can assume that there are no divisors of $V$-lines of the second kind.
We choose  $(\li^1,C^1)$ to be a pair such that $m:=-K_X \cdot \li^1$ is maximum among the anticanonical degrees of families belonging to pairs in $\mathcal B'$; since there are no divisors of the second kind we have $m >1$.\\ 
Since we are assuming that  $\mathcal B'$ contains  a pair $(\li^2,C^2)$ with $[\li^2] \not = [\li^1]$ we have that $m \ge n/2$ by formula (\ref{liml}\!\!).  

\medskip
{\bf Step 1} \quad $E_1 \cdot C^2=0$.\par
\medskip
From the maximality of $m$ it follows that $-K_X \cdot  \li^2 < -K_X \cdot \li^1$, hence that $-K_X \cdot C^2 > -K_X \cdot C^1$. Notice that the numerical class of $C^2$ cannot be proportional to $[\li^1]$, otherwise  $-K_X \cdot C^2=2m \ge n$, and the divisor of $V$-lines associated to the pair $(\li^2,C^2)$  would be of the second kind.\\
Therefore, if $E_1 \cdot C^2 >0$, then for a general $x \in X$ we have, by Lemma (\ref{locvf}) and by Proposition (\ref{iowifam}), that 
$$ \dim \loc(\li^1)_{\loc(C^2)_x} \ge -K_X \cdot C^2 -K_X \cdot \li^1-2 \ge  n,$$
a contradiction, since $\li^1$ is not a covering family.\par
\medskip
{\bf Step 2} \quad The adjoint divisor $D:=K_X +mH$ is nef.\par
\medskip
 If this is not the case, since  $D \cdot \li^1 = 0$ and $D \cdot V >0$, there
 is an extremal ray $R$ on the side of $[\li^1]$ with respect to $[V]$ on which $D$ is negative. Denote by $\f$ the associated contraction and let $W$ be a family of rational curves such that  $\overline{\loc(W)}=\Exc(\f)$ whose degree with respect to $H$ is minimal.\par
 \smallskip
Every fiber of the contraction $\f$ has dimension greater than $m$. If $\f$ is birational then this follows from Proposition (\ref{fiberlocus}), since $l(R) >m$. If else $\f$ is of fiber type
then a general fiber $F$ contains $\loc(W)_x$ for some $x$, hence we have
$$\dim F \ge \dim \loc(W)_x \ge mH\cdot W -1 > m,$$
where the last inequality follows from the fact that $X$ is not covered by lines and that $m >1$.\par
\smallskip
It follows that $E_1 \cap \Exc(\f) = \emptyset$; in fact, if this were not the case, then $E_1$ would meet a fiber of $\f$, and in this case their intersection would contain a curve, contradicting the fact that $R \not \in \conx{E_1}$ (Cf. Lemma (\ref{vdiv})).\\
Therefore $E_1 \cdot R=0$, hence, by Step 1, $[C^2] \in R$. Being $C^2$ a dominating family
$\f$ is a fiber type contraction, contradicting $E_1 \cap \Exc(\f) = \emptyset$.\par
\medskip
{\bf Step 3} \quad The contraction associated to a multiple of $D$ is a smooth blow-up.\par
\medskip
Let $\f:X \to X'$ be the contraction associated to $R:=\mathbb R_+[\li^1]$; let $W$ be a family of rational curves such that  $\overline{\loc(W)}=\Exc(\f)$ whose degree with respect to $H$ is minimal.\\
If $\f$ is of fiber type, then $H \cdot W \ge 2$, since $X$ is not covered by lines; therefore, a general fiber of $\f$ has dimension $\ge 2m-1$. Let $x$ be a general point; then 
$$\dim F + \dim \loc(C^1)_x \ge 2m-1+(n-m) =n+m-1>n, $$
a contradiction.\\
Therefore $\f$ is birational; in particular there exists an irreducible divisor which is negative on $R$ and therefore contains $\Exc(\f)$. So this divisor is $E_1$ and $E_1 \cdot R <0$.
Recall that, by construction we have $E_1 = \loc(\li^1)_{\loc(C^1)_x}$ for a general $x \in X$;
it follows that any fiber $F$  of  $\f$ meets $\loc(C^1)_x$, hence, by Proposition (\ref{iowifam}) we have
$$n \ge \dim F + \dim \loc(C^1)_x \ge \kd \li^1 \kd C^1 -1 = n,$$
hence equality holds. In particular for any fiber of $\f$ we have $\dim F = \kd \li^1$, thus $\f$ is a smooth blow-up by Theorem (\ref{blowup}).\par
\medskip
{\bf Step 4} \quad Conclusion.\par
\medskip
By Step 1 $E_1 \cdot C^2 =0$, hence  $[C^2]$ is in the interior of the
cone $\langle [\li^1], [V]\rangle$ and so $E_1 \cdot V>0$, which in turn implies $E_1 \cdot C^1 \ge 2$.
By formula (\ref{limc}\!\!) we have 
\begin{equation}\label{m+1}
-K_X \cdot C^2 \le n+2 -(-K_X \cdot C^1) = -K_X \cdot \li^1 + 1 = m+1.
\end{equation}
Let $\f:X \to X'$ be the blow-up; the line bundle $H +E_1$ is trivial on $R$, hence there exists $H' \in \pic(X')$ such that  $H+E_1= \f^*H'$. Since $\rho_{X'}=1$ we can write $-K_{X'}= kH'$; being $X'$ Fano and $H+E_1$ nef we have $k>0$.\\
By the canonical bundle formula
$$-K_X = -\f^*K_{X'} -mE_1= kH + (k-m)E_1$$
we have
\begin{equation}\label{2km+1}
m+1 \ge -K_X \cdot C^2 =-\f^*K_{X'} \cdot C^2  =2k
\end{equation}
 and
$$n+1-m=-K_X \cdot C^1 = 2k +(k-m)(E_1 \cdot C^1).$$
Recalling that $E_1 \cdot C^1 \ge 2$; then we have
$$n+1-m \le 2k - 2(m-k)\le 2,$$
where the last inequality follows from (\ref{2km+1}\!).\\
Since $m \le n-1$, the only possibility is that all the inequalities are equalities; in particular $m=n-1$ and  $E_2$ is of the second kind, a contradiction.
\end{proof}
Now we will prove the first part of Theorem (\ref{main}) under the assumption that the Picard number of $X$ is two.
\begin{theorem}\label{thmrho2} Let $(X,H)$ be a RCC-manifold with respect to $V$, not covered by lines, and of Picard number two. Then there exists a polarized manifold $(X',H')$  of Picard number one not covered by lines and an elementary contraction $\f_\sigma:X \to X'$ expressing $X$ as a blow-up of $X'$ along a smooth center $T$, with exceptional divisor $E$, such that  $H \simeq \f_\sigma^*H' - E$. The pair $(X',H')$ is RCC with respect to $V'$, the family of deformations of the image of a general curve parametrized by $V$.
\end{theorem}

\begin{proof}
By Proposition (\ref{proprho2}) we know that, up to numerical equivalence there is only one pair
 $(\li,C)$ in $\mathcal B'$. Let $E$ be the corresponding divisor of $V$-lines. We claim that  $E \cdot V=0$.\\
Assume by contradiction that $E \cdot V >0$; let $x \in X$ be a general point, and consider the following diagram
$$
\xymatrix{\mathcal U_x \ar[d]_p & \ar[l] U_x \ar[r]^i \ar[d]_p & X\\ \V_x &  \ar[l] V_x}
$$
By our assumptions, the inverse image $i^{-1}(E)$ dominates $V_x$; moreover, since $V$ is generically unsplit $i_{|i^{-1}(E)}:i^{-1}(E) \to E$ is a generically finite map, hence it is dominating, since $\dim V_x= \dim i^{-1}(E) = \dim E$.\\
Let $y \in E$ be a general point and let $F$ be a component of $\loc(\li)_y$. We can find a (non complete) curve $\Gamma^0$ in $i^{-1}(F)$; let  $B^0$ be $p(\Gamma^0)$ and $S^0:=p^{-1}(B^0)$; notice that  every curve parametrized by $B^0$ meets $\Gamma^0$.\\
Let $\overline B$ be  the closure of $B^0$ in $\V_x$, let $\overline S$ be $p^{-1}(\overline B)$, let $\nu:B \to \overline B$ be the normalization, $S = B \times_{\overline B} \overline S$ and $\Gamma$ the curve in $S$ whose image in $S^0$ is $\Gamma^0$. Notice that, by construction, the image in $X$ of $\Gamma$ is a curve contained in $F$, hence it is numerically proportional to $[\li]$.\\
By \cite[II.4.19]{Kob} every curve in $S$ is algebraically equivalent to a linear combination
with rational coefficients of a section $C_0$ such that $i(C_0)=x$ and of the irreducible components of fibers of $p_{|S}$ (in \cite[II.4.19]{Kob} take $X = S$, $T= B$ and $Z = C_0$).\\
The images of irreducible components of fibers of $p$ are irreducible curves whose numerical class is $[V]$,  $[\li]$ or $[C]$. If all the fibers of $p:S \to B$ are irreducible, then 
 in $S$ we can write 
 $$\Gamma \equiv \alpha C_0 + \beta f ,$$
 with $[i(f)]=[V]$, and  we get that $i(\Gamma)$ is numerically proportional to $V$, a contradiction. 
So we can write
$$\Gamma \equiv \alpha C_0  +\sum \gamma_i C_i + \sum \delta_j \ell_j,$$
with   $[i(C_i)]=[C]$,  $[i(\ell_j)]=[\li]$, hence
$$[i(\Gamma)] = \sum \gamma_i[C] + \sum \delta_j[\li],$$
and we get $\sum \gamma_i=0$.\\
 Since $i(\Gamma) \not \ni x$  then $\Gamma \cdot C_0 =0$; recalling that $x$ is not contained in any line we get
$$\alpha C_0^2 +\sum \gamma_i = 0;$$
since $C_0$ goes to a point then $C_0^2 <0$, hence $\alpha=0$.\\
This implies that, for a general fiber $f$ we have $f \cdot \Gamma = 0$, a contradiction.\par
\medskip
The statement now follows applying Proposition (\ref{key}).
\end{proof}


\section{Proof of Theorem (\ref{main})}\label{secgen}



\begin{proof} Let $\mathcal B'$  be as in Section 4 and let $\mathcal B=\{(\El^i, C^i)\}_{i=1}^k$ be a maximal set of pairs in $\mathcal B'$  such that $[V],  [\El^1], \dots  [\El^k]$ are numerically independent, and chosen in the following way: if there exists a pair in $\mathcal B'$ whose divisor of $V$-lines is of the second kind we choose it to be $(\li^1,C^1)$; if no such pair exists, but there is a pair whose divisor of $V$-lines is of the third kind we choose it to be $(\li^1,C^1)$.\\
Since we have already proved the first part of the Theorem for $\rho_X=2$ we can assume that $k \ge 2$.\par
\smallskip
We start by showing that all the divisor of the first kind in $\mathcal B$ correspond to blow-ups at points, and can be simultaneously contracted; by Proposition (\ref{key}) this will be the case if $E_i \cdot V=0$ for every such divisor.\\
If $E_1$ is not of the first kind, then $E_i \cdot C^1=0$ for every first kind divisor $E_i$; we already know from Lemma (\ref{disj}) that $E_i \cdot \li^1=0$ for every $V$-divisor in $\mathcal B$ with
$i \not=1$, hence we get that $E_i \cdot V=0$ for every first kind divisor.
So we are left with the case in which every $V$-divisor in $\mathcal B$ (and in $\mathcal B'$, by our choice of $\mathcal B$) is of the first kind. Again we want to prove that $E_i \cdot V=0$ for every $i$.\\
To this aim we will consider the reduction morphism associated to the adjoint divisor $D:=K_X + (n-1)H$, which we claim to be nef and big.\par
\medskip
Assume first that $D$ is not nef; then there exists an extremal ray $R$ of $X$ which has length $l(R)\ge n$; the associated contraction $\f_R$ is of fiber type and has fibers of dimension $\ge n-1$ by Proposition (\ref{fiberlocus}); let $E$ be a first kind divisor and pick $x \in E$; let $F_x$ be the fiber of $\f_R$ passing through $x$. Then 
$$\dim (E \cap F_x) \ge \dim E + \dim F_x -n >0,$$
 a contradiction, since $[\li^1] \not \in R$, being $D \cdot \li^1=0$ and $D \cdot R <0$.\\
Therefore $D$ is nef and defines an extremal face $\tau \subset \cone(X)$, with associated contraction $\f_{\tau}:X \to  Y$.\par
\smallskip
Assume now by contradiction that $D$ is not big, i.e. that $\f_{\tau}$ is of fiber type, and let
$W$ be a minimal dominating family of rational curves such that $[W] \in \tau$.\\
 Then
$ -K_X \cdot W = (n-1)H \cdot W \le n,$
where the last inequality follows from the fact that $W$ is locally unsplit and $\rho_X \ge 1$.
Therefore $H \cdot W=1$, contradicting our assumptions that $X$ is not covered by lines.\par
\medskip
Therefore $D$ is nef and big and we can apply \cite[Theorem 7.3.2]{BSbook} to get that $Y$ is smooth and  $\f_{\tau}$ is the blow-up of $Y$ along $t$ distinct points. Since $D \cdot \li^i=0$ for every $i$ we have $t \ge k$; on the other hand, since $\rho_X=k+1$ we have $t \le k$, so $\f_\tau$  is a blow-up of a smooth $X'$ along $k$ points, and the exceptional divisors are the divisors of $V$-lines.\\
Take a curve $B$ in $X'$ not containing the centers of the blow-up and not meeting the images of the conics parametrized by the families $C^j$ belonging to  pairs in $\mathcal B'$ passing  through a fixed general point $x$. Since all the divisors of $V$-lines are of the first kind there is a finite number of these conics through $x$.\\
By construction the strict transform $\widetilde B$ does not meet cycles in $\V_x$ whose components are not proportional to $V$, hence its numerical class is proportional to $V$. Since $\widetilde B$ does not meet the $E_i$'s we have $E_i \cdot \widetilde B=0$ for every $i$, hence $E_i \cdot V=0$.\par
\medskip
We can thus apply Proposition (\ref{key}) to the set of the first kind divisors, to get a new pair $(X'',H'')$.
If $\rho_{X''} \le 2$ then we are done, otherwise every divisor of $V''$-lines is of the third kind.\\
By Lemma (\ref{disj}) these divisors are disjoint and have intersection number zero with $V$, so we can
apply again Proposition (\ref{key}). In any case we finally get to a pair $(X',H')$ with $\rho_{X'}=1$ and to a family $V'$ as in the statement.\par
\medskip
We now come to the description of $(X',H')$.\\
If $V'$ is a minimal dominating family then $X' \simeq \proj^n$ by \cite[Theorem 1.1]{Kepn};
otherwise there is a dominating family $W$ of rational curves in $X'$ such that $-K_{X'}\cdot W< n+1$.\\
Let $W^*$ be the family of deformation in $X$ of the strict transform of a curve in $W$; 
since $W$ is dominating, a general curve parametrized by $W$ does not meet the union of the centers of the blow-up, $T = \cup T_i$, hence
$E_i \cdot W^* = 0$ for every $i$ and $W^*$ is numerically proportional to $V$.\\
Since $X$ is not covered by lines we can assume that $H \cdot W^*=2$;
using again the canonical bundle formula and the projection formula we get 
$$ -K_{X'} \cdot W= \kd W^* = -\frac{2}{3} K_X \cdot V = \frac{2}{3}(n+1).$$
We know that $H= \f^*H' -\sum E_i$, hence $H' \cdot W=2$ and $\f^*H' \cdot C^i=3$, therefore
$H'$ is the fundamental divisor of $X'$, and the index of $X'$ is $\frac{n+1}{3}$.\par
\medskip
The family $W^*$ is locally unsplit since $X$ is not covered by lines; for a general $x$ we have 
$$\dim \loc(W^*)_x \ge \kd W^* -1 \ge \frac{2n-1}{3}.$$
Notice that $\codim(T_i)-1=-K_X \cdot \li^i$, hence 
$$\dim \loc(C^i)_x =-K_X \cdot C_i -1= n +K_X \cdot \li^i = \dim T_i +1.$$
If for some $i$ we have  $\dim T_i \ge (n+1)/3$ then
 we get a contradiction considering the intersection $\loc(C^i)_x \cap \loc(W^*)_x$ for any $i$.
\end{proof}


\section{RCC Fano  manifolds}\label{secfano}


In this section we will show how restricting to RCC Fano manifolds leads to
stronger results.

\begin{proposition}\label{target} In the assumptions of Theorem (\ref{main}), if $X$ is a Fano manifold and
 $(X',H') \not\simeq (\proj^n, \Ol(3))$ then  $\dim T_i=\frac{n-2}{3}$ for every $i$.
\end{proposition}

\begin{proof} 
We keep the notation of the proof of Theorem (\ref{main}).\\
The rc$(\W^*)$-fibration contracts curves parametrized by $V$, hence it goes to a point. Therefore
$W^*$ cannot be a quasi-unsplit family, otherwise $\rho_X=1$, hence there is a reducible cycle $l^* + \bar l^*$ in $\W^*$ whose components are not numerically proportional.\\
For at least one $i$ the divisor $E_i$ is not trivial on both $l^*$ and $\bar l^* $, hence, up to exchange them, we can assume $E_i \cdot l^* <0$; therefore $[l^*] \in R^i$ and $-K_X \cdot l^* = -K_X \cdot \li^i$, so
$$-K_X \cdot \li^i < -K_X \cdot W^* = \dfrac{2(n+1)}{3},$$
and we get  $\codim T_i \le 2/3(n+1)$, which, combined with the bound obtained in Theorem (\ref{main}) gives $\dim T_i=\frac{n-2}{3}$.
\end{proof}

\subsection{Higher Picard number}

In this section we are going to prove that a Fano RCC manifold has Picard number at most three, and to classify those of Picard number three.\\
We will need the following result from basic projective geometry; the proof we are giving here was pointed out to us by Francesco Russo.

\begin{lemma}\label{sec} Let $T \subset \proj^n$ be a projective manifold, and denote by $\mathcal S(T)$ its
secant variety. Then
\begin{enumerate}
\item If $\dim \mathcal{S}(T)=\dim T$ then $T$ is a linear subspace of $\mathbb{P}^n$.
\item If $\dim \mathcal{S}(T)=\dim T+1$ then $T$ is a hypersurface of $\langle T \rangle  \simeq\mathbb{P}^{\dim T+1}$.
\end{enumerate}
\end{lemma}

\begin{proof} We denote by $\mathbb T_z Z$ the projective tangent space to a variety $Z$ at a point $z$; if $t \in T$ we denote by $\mathcal{S}(t,T)$ the relative secant variety of $T$ with respect to $t$.\\
First of all notice that for every $t\in T$
\begin{equation}\label{secant}
T\subseteq \mathcal{S}(t,T)\subseteq \mathbb T_t\mathcal{S}(t,T)\subseteq \mathbb T_t \mathcal{S}(T).
\end{equation}
Assume that $\dim \mathcal{S}(T)=\dim T$; clearly $\mathcal{S}(T)=T$. Let $t\in T$ be a point of $T$. By (\ref{secant}\!) and by our assumptions we have that
$$T\subseteq \mathbb T_t \mathcal{S}(T)=\mathbb T_t T=\mathbb{P}^{\dim T}$$
and hence $T=\mathbb{P}^{\dim T}$ since $T$ and $\mathbb T_t T$ are  irreducible varieties of the same dimension $\dim T$.\\
Now we suppose that $\dim \mathcal{S}(T)=\dim T+1$. For a general point $t\in T$ 
$$T\subsetneq\mathcal{S}(t, T)\subseteq\mathcal{S}(T)$$
and hence
$$\dim T< \dim \mathcal{S}(t, T)\leq\dim \mathcal{S}(T)=\dim T+1.$$
This implies that for a general point $t\in T$ we have $\mathcal{S}(t, T)=\mathcal{S}(T)$, hence for a general point $x\in \mathcal{S}(T)\setminus T$ there exists $t'\in T$ such that 
$$x\in\langle t,t'\rangle\subset \mathcal{S}(t,T)=\mathcal{S}(T).$$ 
From this it follows that a general point $t\in T$ is contained in $\mathbb T_x \mathcal{S}(T)$ and that
$$\mathcal{S}(T)\subseteq \langle T \rangle\subseteq \mathbb T_x \mathcal{S}(T),$$
where $\langle T\rangle$ is the linear span of $T$ in $\mathbb{P}^n$.\\
Now, by the generality of $x\in \mathcal{S}(T)$ we know that $\dim \mathbb T_x \mathcal{S}(T)=\dim \mathcal{S}(T)$ and hence
$$\mathcal{S}(T)=\mathbb T_x \mathcal{S}(T)=\mathbb{P}^{\dim \mathcal{S}(T)}=\mathbb{P}^{\dim T+1},$$
which concludes the proof.
\end{proof}

\begin{proof}[Proof of Theorem (\ref{th:15})] By Theorem (\ref{main}) and Proposition (\ref{target}) we know that $(X,H)$ is the blow-up of $(X', H')$ along disjoint centers. Assume that $\rho_X >3$ and let $(\li^i,C^i), (\li^k,C^k)$ be two independent pairs in $\mathcal B'$ with associated divisors of $V$-lines $E_i$ and $E_k$.\\
Since $E_i$ cannot contain curves of $C^k$, but $C^k$ is dominating, it follows that there exists a reducible
cycle $l_k + \bar l_k$ in $\mathcal C^k$ such that $E_i \cdot l_k <0$; this implies that $[l_k] \in \conx{E_i} \subset \Pi_i$.
Notice that $H+E_i$ is nef on $\Pi_i$, hence $E_i \cdot l_k=-1$; since both $H$ and $E_i$ have the same intersection number with $l_k$ and $\El^i$ then $[l_k]=[\El^i]$.\\
Recalling that $X$ is Fano, this implies that $-K_X\cdot C^k\geq-K_X\cdot	\El^i+1$, and so we get that $\kd (\El^k + \El^i) \le n$. Hence, by formula (\ref{liml}\!\!) that
\begin{equation}
\kd (\El^k + \El^i) = n.
\end{equation}
Notice that $-K_X \cdot \li^i$ is the length of the contraction of the extremal ray generated by $[\li^i]$; in particular, if $T_i$ is the image of $E_i$ via $\pi$, we have $\dim T_i = n - (-K_X \cdot \li^i)-1$,
hence, for $i \not = k$
\begin{equation}\label{dimt}
\dim T_k + \dim T_i = n -2.
\end{equation}
By Proposition (\ref{target}) it follows that  $(X',H') \simeq (\proj^n, \Ol(3))$.\\
If $\rho_X > 3$, combining with formula (\ref{limc}\!\!) we have that,
for every $i$,
$$\kd C^i = \dfrac{n+2}{2} \quad \text{and} \quad \kd \El^i = \dfrac{n}{2} ,$$ 
hence $\dim T_i = \frac{n-2}{2}$ for every $i$.\par
\medskip 
Recall now that $H=\pi^*\Ol(3) - \sum E_i$ is ample; take two of the centers $T_1$ and $T_2$
and consider their join: it has dimension $n-1$, hence it meets some other center $T_3$.
Take a line $\ell$ meeting three centers; then $(\pi^*\Ol(3) - \sum E_i) \cdot \ell \le 0$,
a contradiction. Therefore $\rho_X\leq 3$.\par
\medskip
Assume now that $\rho_X=3$.
Let $\mathcal{S}(T_1)$ be the secant variety of $T_1$. Suppose that $\dim\mathcal{S}(T_1)\geq\dim T_1+2$. Then 
\begin{eqnarray*}
\dim (\mathcal{S}(T_1)\cap T_2)&\geq&\dim \mathcal{S}(T_1)+\dim T_2-n\\
&\geq&\dim T_1+2+\dim T_2-n\\
&=&0
\end{eqnarray*}
i.e. there is a line $l$ in $\mathbb{P}^n$ which meets $T_1$ in two points and $T_2$ in a point; as above we show that $H \cdot l \le 0$.\\
It follows by Lemma (\ref{sec}) that either $T_i$ is a linear space or an hypersurface in a linear space of dimension $\dim T_i +1$.\\ 
Notice also that from the ampleness of $H$ it follows that there cannot exist trisecant lines of $T_i$ in $\mathbb{P}^n$, and hence, if it is not a linear space then $T_i$ is a hyperquadric, and we can prove that $\dim T_i\geq\frac{n}{2}-1$.\\ In fact, considering the strict transform $l$ of a secant line of $T_i$ and recalling that $X$ is Fano, by the canonical bundle formula, we get
$$
1\leq-K_X\cdot l= (n+1)-2\codim(T_i) +2
$$
$$\Rightarrow \dim T_i\geq \frac{n}{2}-1.$$
Therefore
\begin{enumerate}
	\item if $\dim \mathcal{S}(T_1)=\dim T_1$ and $\dim \mathcal{S}(T_2)=\dim T_2$, then $X$ is the blow-up of $\mathbb{P}^n$ along two disjoint linear subspaces.
	\item If $\dim \mathcal{S}(T_1)=\dim T_1$ and $\dim \mathcal{S}(T_2)=\dim T_2+1$, then $X$ is the blow-up of $\mathbb{P}^n$ along a linear subspace $T_1$ and along a quadric $T_2\subset\Lambda_2\simeq\mathbb{P}^{\dim T_2+1}$ such that $\dim T_1\leq(\frac{n}{2}-1)$. Moreover $\Lambda_2$ and $T_1$ must be disjoint, because there cannot exist lines in $\mathbb{P}^n$ which meet $T_1$ in a point and $T_2$ in two points. 
	\item If $\dim \mathcal{S}(T_1)=\dim T_1+1$ and $\dim \mathcal{S}(T_2)=\dim T_2+1$, then $X$ is the blow-up of $\mathbb{P}^n$ along two quadrics $T_1,T_2$ such that $\dim T_1=\dim T_2=\left(\frac{n}{2}-1\right)$ (clearly $n$ is even). Notice also that $T_i\subset\Lambda_i\simeq\mathbb{P}^{\frac{n}{2}}$, and $\dim (\Lambda_1\cap\Lambda_2)=0$ because there cannot exist trisecant lines of $T_1\cup T_2$. 
\end{enumerate}
\end{proof}

\subsection{Picard number two: some examples}

Let $(X,H)$ be a RCC Fano manifold of Picard number two obtained as the blow-up of  $\proj^n$ along a smooth center $T$. Denote by $\f:X \to \proj^n$ the blow-up contraction and by $E$ the exceptional divisor; then $H =\f^*\Ol_{\proj^n}(3) -E$.\\
By the ampleness of the anticanonical bundle
$$-K_X = \f^*\Ol_{\proj^n}(n+1) - (\codim T-1)E,$$ 
we have that, if $T$ is not a linear space, then $\dim T > (n-3)/2$. To see this, just compute
the intersection of $-K_X$ with the strict transform of a secant line of $T$.\\ 
Moreover, by the ampleness of $H$ we get that $T$ has no trisecants.
A large class of examples is given by the following

\begin{proposition}\label{Fano2} Let $T \subset \proj^n$ be a smooth subvariety of dimension $t > (n-3)/2$ whose homogeneous ideal is generated by quadrics. Then the pair $(X,H)=(Bl_T(\proj^n), 3\cH-E)$ is a Fano RCC manifold.
\end{proposition}

\begin{proof} Consider the rational map $\psi: \rc{\proj^n}{\proj^N}$ given by a system of quadrics which generates $\mathcal I(T)$ and the resolution of this map:
$$\xymatrix@=15pt@R=40pt{ & X=Bl_T(\proj^n) \ar[ld]_\f \ar[rd]^{\widetilde \psi} & \\
\proj^n \ar@{-->}[rr]^\psi & & \proj^N}$$
The morphism $\widetilde \psi$  is given by the linear system $|\f^*\Ol_{\proj^n}(2)-E|$, hence it contracts the strict transforms of the (bi)secants to $T$; using the canonical bundle formula we see that the intersection number of  $-K_X$ with these curves is positive, therefore $X$ is a Fano manifold.\\
The ampleness of $\f^*\Ol_{\proj^n}(3)-E$ is now given by the Kleiman criterion. The family $V$ is the family of deformation of the strict transform of a general line in $\proj^n$.
\end{proof}

\begin{remark}
Some examples of manifolds obtained as in Proposition (\ref{Fano2}) can be found in \cite{CO2}, Cases (b1)-(b6) and (c1)-(c2).
\end{remark}

\section{Examples}

\begin{example}
 $(X,H)\simeq(Bl_{\Lambda_1,\Lambda_2}(\mathbb{P}^n),3\mathcal{H}-E_1-E_2)$, where $Bl_{\Lambda_1,\Lambda_2}(\mathbb{P}^n)$ is the blow-up of $\mathbb{P}^n$ along two disjoint linear spaces  $\Lambda_1\simeq\mathbb{P}^{t}$ and $\Lambda_2\simeq\mathbb{P}^{n-2-t}$;	 $E_1,E_2$ are the exceptional divisors of  $\pi$ and $\mathcal{H}=\pi^*\mathcal{O}_{\mathbb{P}^n}(1)$.
Denote by 
\begin{itemize}
\item $R^i$ the extremal ray corresponding to the contraction of $E_i$;
\item $\varepsilon_i$ the contraction associated to $R^i$;
\item $\ell_i$ a minimal curve whose numerical class is in $R^i$;
\item $\ell$ a curve which is the strict transform of a line meeting both $\Lambda_1$ and $\Lambda_2$ in a point;
\item $D_i=\mathcal{H}-E_i$;
\end{itemize}
The line bundles  $\mathcal{H},{D}_1$ and ${D}_2$ are nef  on $X$; the cone of curves is therefore contained in the intersection of the positive halfspaces of $\cycl(X)$ determined by them.
By looking at the intersection numbers with the curves $\ell_1, \ell_2, \ell$:\par
\medskip
\begin{center}
\begin{tabular}{c|c|c|c}
	&$\ell_1$&$\ell_2$&$\ell$\\
	\hline
	$\mathcal{H}$&$0$&$0$ &$1$\\
	\hline
	${D}_1$&$1$&$0$ &$0$\\
	\hline
	${D}_2$&$0$&$1$ &$0$\\
	 \hline
  \end{tabular}
\end{center}\par
\medskip
\noindent we see that $\cone(X)$ is the intersection of those halfspaces, and that is spanned by three rays,
$R^1 = \mathbb R_+[\ell_1]$, $R^2 = \mathbb R_+[\ell_2]$, $R^3 = \mathbb R_+[\ell]$.
Clearly the elementary contractions associated to $R^1$ and $R^2$ are the the blow-downs of $E_1$ and $E_2$.\\ The elementary contraction associated to $R^3$ is divisorial, and its exceptional locus is the strict transform of the join $J(\Lambda_1,\Lambda_2)$;
it is possible to show that this contraction is the blow-up of $\proj^{n-t-1} \times \proj^{t+1}$
along a smooth subvariety $\proj^{n-t-2} \times \proj^{t}$.\par
\smallskip
\noindent
\textbf{Description the families of rational curves}
\vskip1mm
\noindent
In this example the family $V$ of cubics is the family of deformations of the strict transform of a general line of $\mathbb{P}^n$; the set $\mathcal B'$ consists of two pairs, $(\li^1,C^1)$ and $(\li^2,C^2)$: the families $\li^i$ are the families of lines contracted by the blow-down, while the families $C^i$ are the families of strict transforms of lines in $\proj^n$ meeting one of the centers.\\
Curves in $C^i$ degenerate into a line contracted by $\varepsilon_j$ ($i \not = j$) and the strict transform of a line meeting both $\Lambda_1$ and $\Lambda_2$.
\begin{center}
\includegraphics[width=5.5cm]{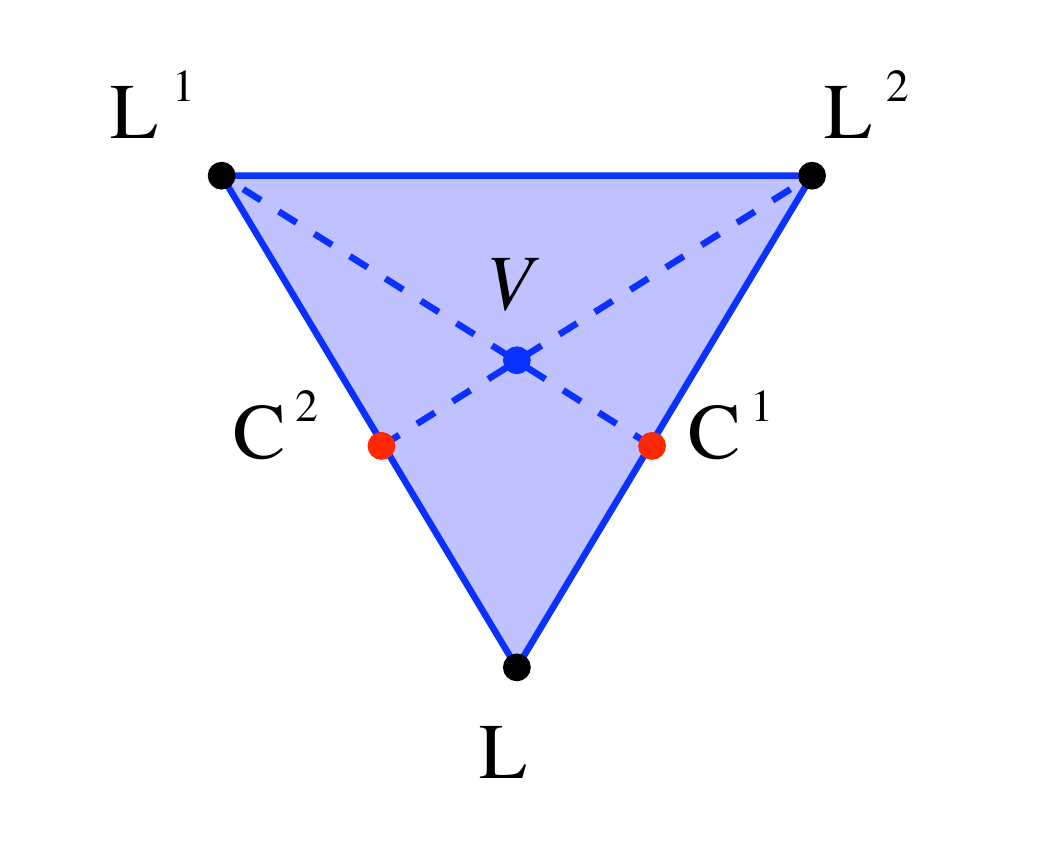}
\end{center}

\end{example}

\begin{example}\label{lq}
 $(X,H)\simeq(Bl_{\Lambda_1,Q_1}(\mathbb{P}^n),3\mathcal{H}-E_1-E_2)$, where $Bl_{\Lambda_1,Q_1}(\mathbb{P}^n)$ is the blow-up of $\mathbb{P}^n$ along a linear space $\Lambda_1\simeq\mathbb{P}^{t}$ and a smooth quadric $Q_1\subset\Lambda_2\simeq\mathbb{P}^{n-1-t}$ such that $\Lambda_1 \cap \Lambda_2 = \emptyset$;	 $E_1,E_2$ are the exceptional divisors of  $\pi$ and $\mathcal{H}=\pi^*\mathcal{O}_{\mathbb{P}^n}(1)$.
Denote by 
\begin{itemize}
\item $R^i$ the extremal ray corresponding to the contraction of $E_i$;
\item $\varepsilon_i$ the contraction associated to $R^i$;
\item $\ell_i$ a minimal curve whose numerical class is in $R^i$;
\item $\overline \ell_1$ a curve which is the strict transform of a line meeting both $\Lambda_1$ and $Q_1$ in a point;
\item $\overline \ell_2$ a curve which is the strict transform of a general line contained in $\Lambda_2$; 
\item $D_1=\mathcal{H}-E_1$;
\item $D_2=2\mathcal{H}-E_2$;
\item $D_3=2\mathcal{H}-E_1-E_2$.
\end{itemize}
The line bundles  $\mathcal{H},{D}_1$ and ${D}_2$ are nef  on $X$; we want to show that also $D_3$ is nef.\\
Suppose by contradiction that there is a irreducible curve $C\subset X$ such that $D_3\cdot C <0$. Then $(\mathcal{H}-E_1)\cdot C<0$ or $(\mathcal{H}-E_2)\cdot C <0$.\\ 
Assume  that $(\mathcal{H}-E_2)\cdot C<0$ (the other case is dealt with in a similar way); the map $\pi$ factors as $\varepsilon_2 \circ \varepsilon_1$; let $\widetilde{H}$ be an hyperplane of $\mathbb{P}^n$ which contains $\Lambda_2$ and let $H'$ be the strict transform of $\widetilde{H}$ via $\varepsilon_2$. We have that
$$\mathcal{H}-E_2=\varepsilon_1^* H'$$
and hence, by the projection formula, we get
$$(\mathcal{H}-E_2)\cdot C=H'\cdot {\varepsilon_1}_{*}C<0.$$
This implies that the curve $C$ is not contracted by $\varepsilon_1$ and that $\varepsilon_1(C)$ is contained in $H'$. Since this holds for every hyperplane containing $\Lambda_2$ we have $\pi(C)\subset \Lambda_2$, so either $C$ is contained in the strict transform of $\Lambda_2$ or $C \subset E_2$.\\
Since $E_2=\proj(\mathcal N_{Q_1/\proj^n}^*)\simeq \proj (\Ol(-2) \oplus \Ol(-1)^{\oplus n-2-t})$, with the strict transform of  $\Lambda_2$ cutting the section corresponding
to the surjection $\mathcal N_{Q_1/\proj^n}^* \to \Ol(-2)$ we have that $\conx{E_2}=\langle [\overline \ell_2], [\ell_2] \rangle$, while every curve contained in the strict transform of $\Lambda_2$ is numerically proportional to $[\overline\ell_2]$.\\
Since $D_3 \cdot \ell_2 =1$ and $D_3 \cdot \overline \ell_2 =0 $ we get a contradiction
which proves the nefness of $D_3$.\\
We have four nef line bundles: $\mathcal H, D_1,D_2$ and $D_3$; the cone of curves is therefore contained in the intersection of the positive halfspaces of $\cycl(X)$ determined by them.
By looking at the intersection numbers with the four curves $\ell_1, \ell_2, \overline \ell_1,  \overline \ell_2$:\par
\medskip
\begin{center}
\begin{tabular}{c|c|c|c|c}

	&$\ell_1$&$\ell_2$&$\overline \ell_1$&$\overline \ell^{{~}^{~}}_2$\\
	\hline
	$\mathcal{H}$&$0$&$0$ &$1$ &$1$\\
	\hline
	${D}_1$&$1$&$0$ &$0$ &$1$\\
	\hline
	${D}_2$&$0$&$1$ &$1$ &$0$\\
	 \hline
	$D_3$&$1$&$1$ &$0$ &$0$\\
    \hline
  \end{tabular}
\end{center}\par
\medskip
\noindent we see that $\cone(X)$ is the intersection of those halfspaces, that is spanned by four rays,
$R^1 = \mathbb R_+[\ell_1]$, $R^2 = \mathbb R_+[\ell_2]$, $R^3 = \mathbb R_+[\overline \ell_1]$,
$R^4 = \mathbb R_+[\overline \ell_2]$ and that the position of these rays is as in the next figure, which shows a cross section of $\cone(X)$.
\begin{center}
\includegraphics[width=6.5cm]{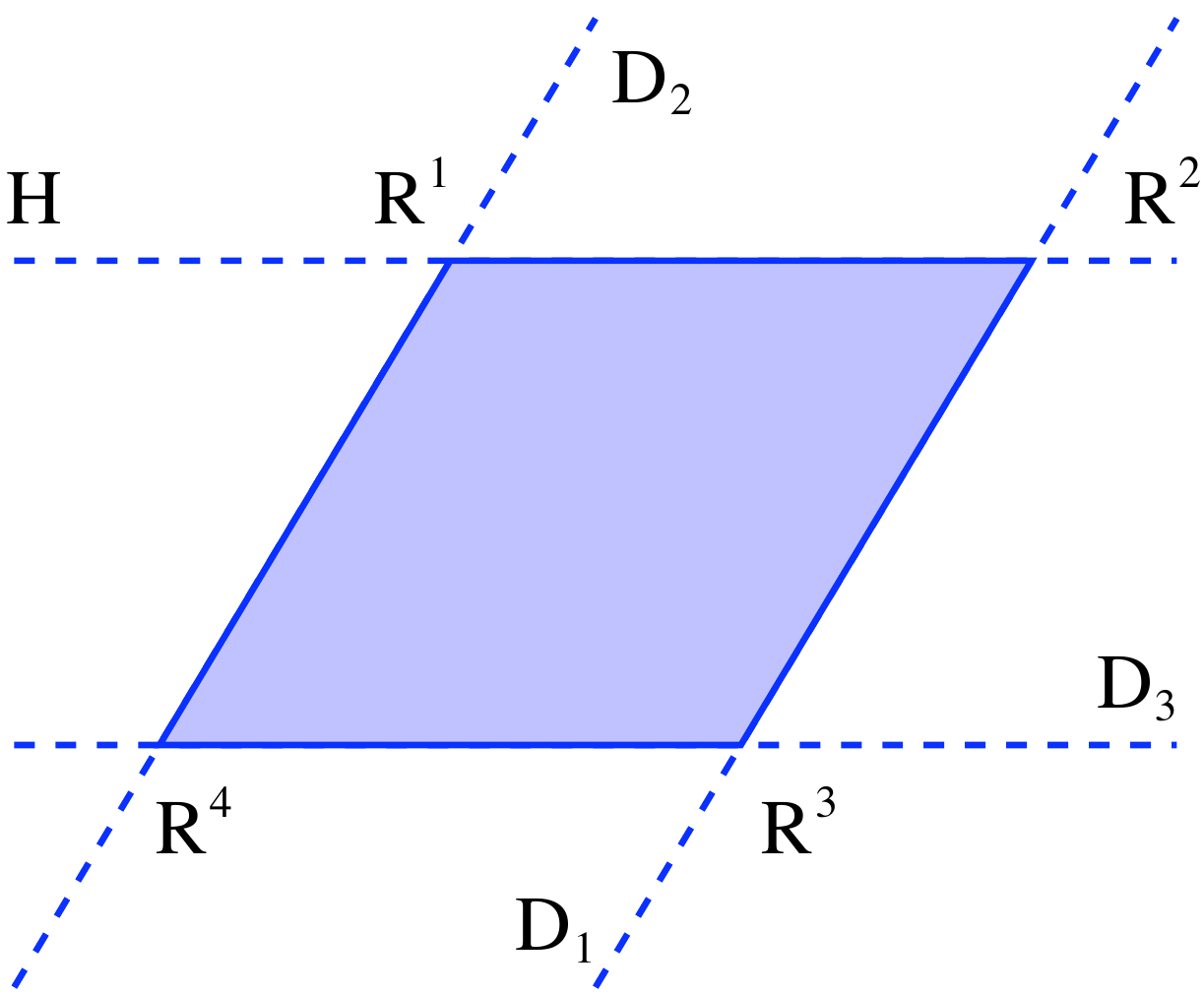}
\end{center}
Clearly the elementary contractions associated to $R^1$ and $R^2$ are the the blow-downs of $E_1$ and $E_2$. The elementary contraction associated to $R^3$ is divisorial, and its exceptional locus is the strict transform of the join $J(\Lambda_1,Q_2)$, which is a divisor linearly equivalent to $2\mathcal{H} -2E_1-E_2$. Moreover, the contraction associated to $R^4$ contracts the strict transform of $\Lambda_2$; hence if $\dim \Lambda_2=n-1$ this contraction is divisorial, otherwise it is small.\par
\smallskip
\noindent
\textbf{Description the families of rational curves}
\vskip1mm
\noindent
In this example the family $V$ of cubics is the family of deformations of the strict transform of a general line of $\mathbb{P}^n$; the set $\mathcal B'$ consists of two pairs, $(\li^1,C^1)$ and $(\li^2,C^2)$: the families $\li^i$ are the families of lines contracted by the blow-down, the family $C^1$ is the family of strict transforms of lines in $\proj^n$ meeting $\Lambda_1$ at one point and the family $C^2$ is the family of strict transforms of lines in $\proj^n$ meeting $Q_1$ at one point.\\
Curves parametrized by  $C^1$ degenerate into the strict transform of a line meeting $\Lambda_1$ and $Q_1$ and a line  contracted by $\varepsilon_2$, while curves parametrized by $C^2$ degenerate in two possible ways: either as a line contracted by $\varepsilon_2$ and the strict transform of a line contained in $\Lambda_2$ or as a line contracted by $\varepsilon_1$  and the strict transform of a line meeting both $\Lambda_1$ and $Q_1$. \begin{center}
\includegraphics[width=6cm]{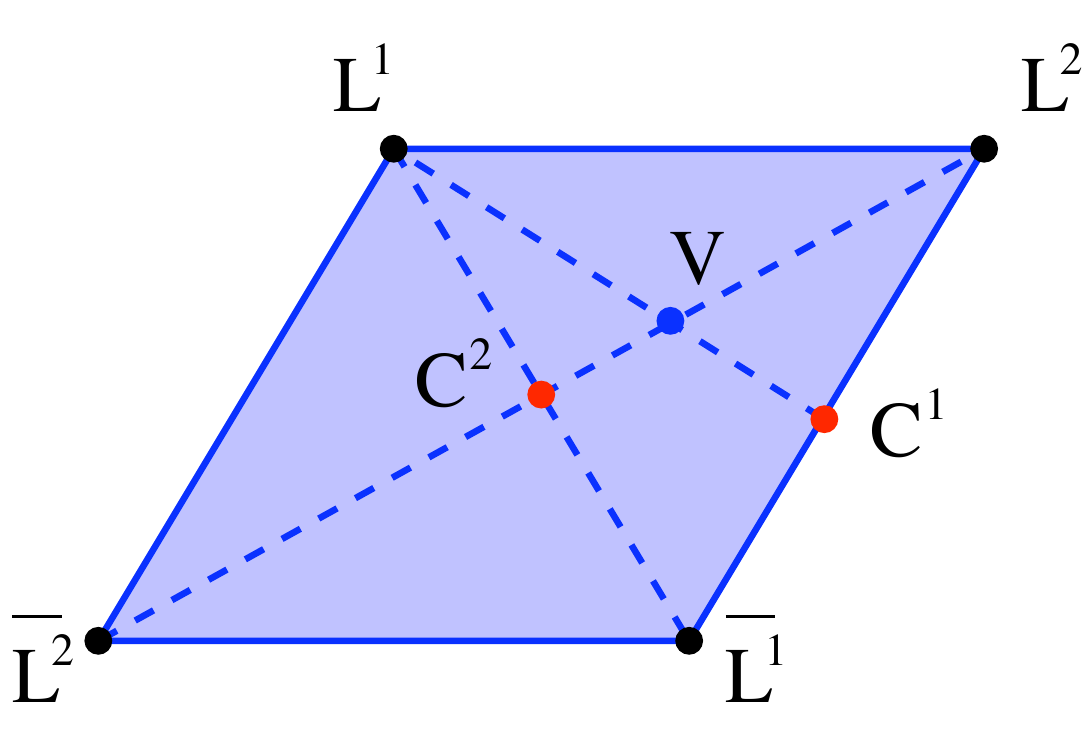}
\end{center}
\end{example}

\begin{example}
 $(X,H)\simeq(Bl_{Q_1,Q_2}(\mathbb{P}^n),3\mathcal{H}-E_1-E_2)$, where $Bl_{Q_1,Q_2}(\mathbb{P}^n)$ is the blow-up of $\mathbb{P}^n$ along two smooth quadrics  $Q_1\subset\Lambda_1\simeq\mathbb{P}^{\frac{n}{2}}$ and $Q_2\subset\Lambda_2\simeq\mathbb{P}^{\frac{n}{2}}$ such that 
$$\begin{array}{lll}
Q_1 \cap Q_2 = \emptyset,&
\dim \Lambda_1\cap\Lambda_2=0,&
\dim Q_1=\dim Q_2=\frac{n}{2}-1,
\end{array}$$		
$E_1,E_2$ are the exceptional divisors of  $\pi$ and $\mathcal{H}=\pi^*\mathcal{O}_{\mathbb{P}^n}(1)$.
Denote by 
\begin{itemize}
\item $R^i$ the extremal ray corresponding to the contraction of $E_i$;
\item $\varepsilon_i$ the contraction associated to $R^i$;
\item $\ell_i$ a minimal curve whose numerical class is in $R^i$;
\item $\overline \ell_i$ a curve which is the strict transform of a line  contained in $\Lambda_i$;
\item $D_i=2\mathcal{H}-E_i$;
\item $D_3=2\mathcal{H}-E_2-E_1$.
\end{itemize}
\noindent  As  in Example (\ref{lq}) we can show that $\cone(X)$ is the intersection of the halfspaces determined by the nef divisors $\mathcal{H},D_1,D_2,D_3$, that is spanned by four rays,
$R^1 = \mathbb R_+[\ell_1]$, $R^2 = \mathbb R_+[\ell_2]$, $R^3 = \mathbb R_+[\overline \ell_1]$,
$R^4 = \mathbb R_+[\overline \ell_2]$ and that the position of these rays is as in the next figure, which shows a cross section of $\cone(X)$.
\begin{center}
\includegraphics[width=7.5cm]{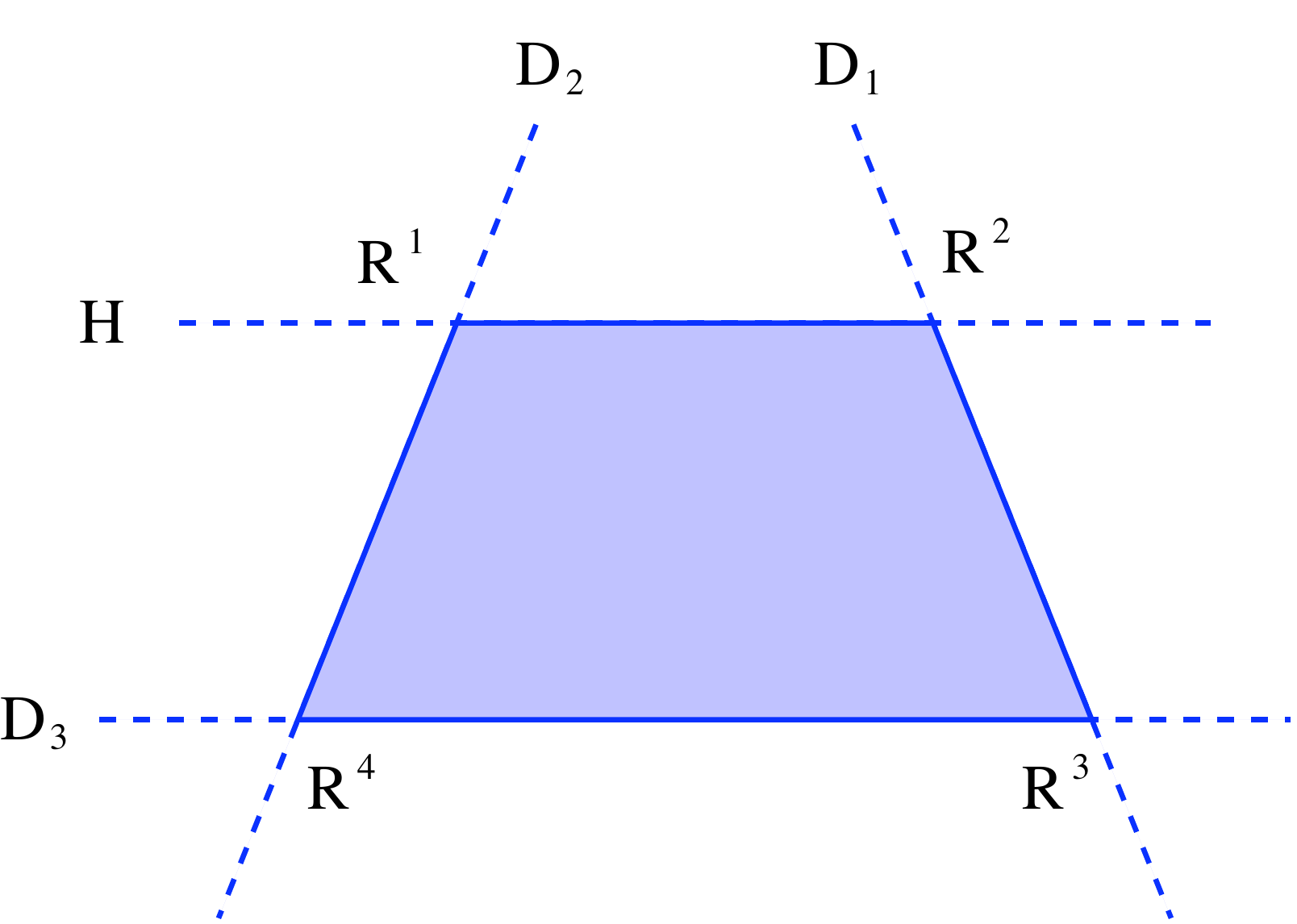}
\end{center}
Clearly the elementary contractions associated to $R^1$ and $R^2$ are the the blow-downs of $E_1$ and $E_2$. The elementary contractions associated to $R^3$ and to $R^4$ are small with exceptional loci of dimension $\frac{n}{2}$ which are the strict transforms of the linear spaces $\Lambda_i$.\\
Let us spend a few words on the contraction  of the face $\sigma =\langle R^3, R^4 \rangle$.\\
Let $P$ be the intersection point of $\Lambda_1$ and $\Lambda_2$, and let $\Sigma$ be a $2$-plane passing through $P$ and meeting $\Lambda_1$ and $\Lambda_2$ in two lines, $l_1$ and $l_2$. It is not difficult to see that for a general point of $\proj^n$ there is exactly one such $2$-plane.\\
For a general point $Q$ of $\Sigma$ there is a conic passing through $Q$ and through the (four) points of intersections of $\Sigma$ with $Q_1$ and $Q_2$; denote by $\gamma$ this conic. Since $\gamma$ meets both $Q_1$ and $Q_2$ in two points we have $E_1 \cdot \gamma=E_2 \cdot \gamma=2$.\\
The contraction $\f:X \to Y$ of $\sigma$ is supported by $D_3=2\mathcal H -E_1 -E_2$, hence it contracts $\gamma$; it follows that $\f$ is of fiber type. The restriction of $D_3$ to $E_1$ is big, hence $\dim Y=n-1$.\\
Therefore $\f$ is a conic bundle; the divisor of reducible conics is the strict transform of the join $J(Q_1,Q_2)$, and there is one special fiber of dimension $n-2$ consisting of two irreducible components which are projective spaces meeting at a point, namely the strict transforms of $\Lambda_1$ and $\Lambda_2$.\par
\smallskip
\noindent
\textbf{Description the families of rational curves}
\vskip1mm
\noindent
In this example the family $V$ of cubics is the family of deformations of the strict transform of a general line of $\mathbb{P}^n$; the set $\mathcal B'$ consists of two pairs, $(\li^1,C^1)$ and $(\li^2,C^2)$: the families $\li^i$ are the families of lines contracted by the blow-down, while the families $C^i$ are the families of strict transforms of lines in $\proj^n$ meeting $Q^i$ at one point.\\
Curves in $C^i$ degenerate in two possible ways: either as a line contracted by $\varepsilon_i$ and the strict transform of a line contained in $\Lambda_i$ or as a line contracted by $\varepsilon_j$ ($i \not = j$) and the strict transform of a line meeting both $Q_1$ and $Q_2$. This last curves are numerically proportional to the conics meeting both
$Q_1$ and $Q_2$ in two points.\\
\begin{center}
\includegraphics[width=6cm]{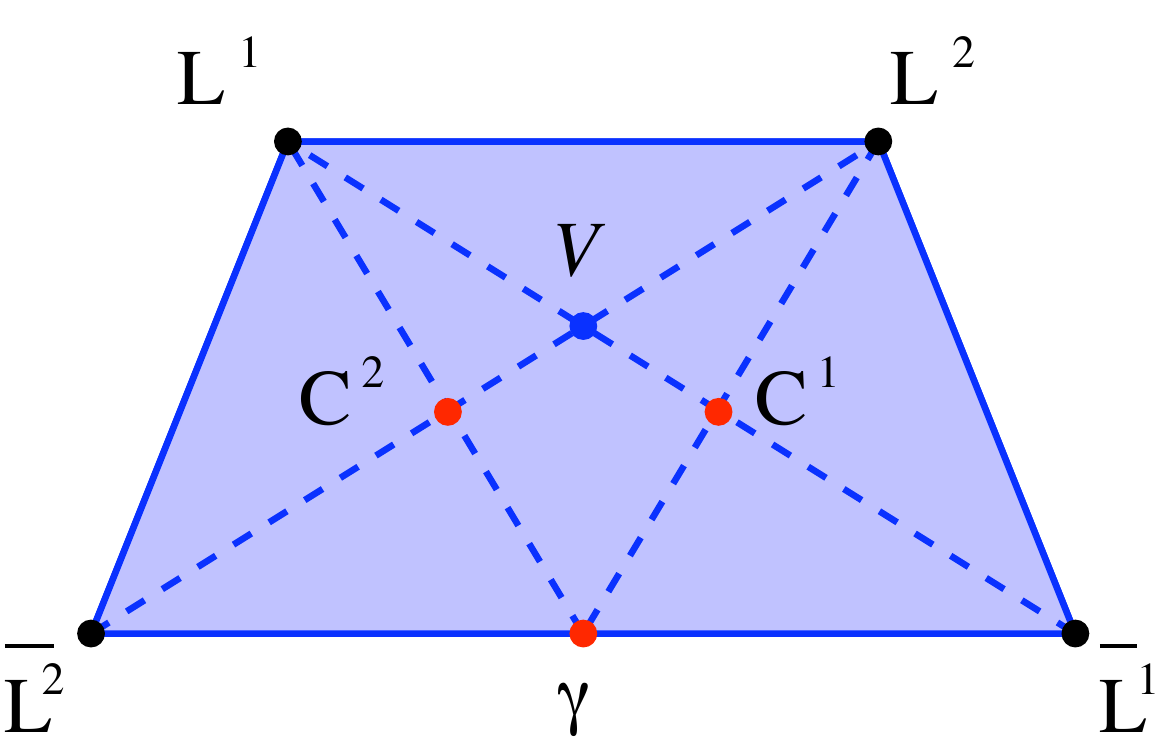}
\end{center}

\end{example}


\bigskip
\noindent
\small{{\bf Acknowledgements}. 
We would like to thank Francesco Russo and Massimiliano Mella for many useful comments and suggestions.\par
\medskip
\noindent
\small{{\bf Note}. 
This work grew out of a part of the second author's Ph.D. thesis \cite{P} at the
Department of Mathematics of the University of Trento.\par

\end{document}